\documentclass[12pt]{article}
\usepackage[utf8]{inputenc}     
\usepackage[english]{babel}     
\usepackage[a4paper, left=1.5in, right=1.5in, top=1.5in, bottom=1.5in]{geometry}
\usepackage{graphicx}           
\usepackage{amsmath,amssymb}    
\usepackage{amsthm}             
\usepackage{mathtools}          
\usepackage{enumitem}           
\usepackage{listings}           
\usepackage{todonotes}          
\usepackage{hyperref}           
\usepackage{indentfirst}
\usepackage{amsthm}
\usepackage{pifont}
\usepackage{mathrsfs}

\def\thesistitle{Estimate of Periodic Orbits of Degenerate Hamiltonians}
\def\thesisauthorfirst{Jiao Hao}
\def\thesisauthorsecond{}
\def\thesisdate{December 2026}

\title{\thesistitle}
\author{\thesisauthorfirst\space\thesisauthorsecond}
\date{\thesisdate}

\newtheorem{theorem}{Theorem}[section]

\newtheorem{lemma}[theorem]{Lemma}
\newtheorem{proposition}[theorem]{Proposition}

\theoremstyle{definition}
\newtheorem{definition}[theorem]{Definition}
\theoremstyle{claim}
\newtheorem{claim}[theorem]{Claim}

\theoremstyle{remark}

\begin{document}
\setlength{\parindent}{2em}

\newpage
\maketitle

At first, I want to mention that this is my graduation thesis as an undergraduate student. In this paper, I provide a survey of some results of degenerate Arnold conjecture and write down some of them in the language of Floer theory explicitly, so there are no new results in this paper.

\section{Introduction}

Morse theory is a useful tool for studying critical points of functions. Using gradient flow, it cleverly combines the number of critical points with the homology groups. So, people want to use the same strategy to study the property of time-dependent things.

That is, for any $H:S^1\times M\rightarrow\mathbb{R}$, we study the gradient flow of actional function $\mathcal{A}_H$ defined on $\Omega_0(M)$, the connected component contains constant loop in loop sapce of M. $\mathcal{A}_H$ is defined by $\mathcal{A}_H(\gamma):=\int_{S^1}H_t(\gamma(t))dt+\int_{\mathbb{D}^2}u^*\omega$ for some $u:\mathbb{D}^2\rightarrow M$ with $u|_{\partial \mathbb{D}^2} = \gamma$. Note that this definition is depend on the choice of u, so $\mathcal{A}_H$ can only be a $S^1$ valued function or well defined on the universal covering of $\Omega_0(M)$.

A direct observation is that $\gamma$ is a critical point of $\mathcal{A}_H$ if and only if it satisfies the equation

\begin{center}
    $\frac{d\gamma}{dt} = X_{H_t}(\gamma(t))$.
\end{center}
$X_{H}$ is defined by $dH = \omega(\cdot,X_H)$. And this is just the formula of Hamiltonian flow with Hamiltonian H. So, critical points of $\mathcal{A}_H$ are correspondence to fixed point of time-1 Hamiltonian transformation $\psi^1$, $\dot{\psi^t} = X_{H_t}(\psi^t)$ $\psi:M\rightarrow M$.

\begin{center}
    \{ critical points of $\mathcal{A}_H$ \}$\longleftrightarrow$\{ fixed points of $\psi^1$ \}\\
    $\ \ \ \ \ \ \ \ \ \gamma \ \ \ \ \  \longleftrightarrow\ \ \ \ \ \gamma(0)$
\end{center}

Then about the number of these fixed points, Vladimir Arnold proposed Arnold conjecture in the 1960s\cite{Arnold86}. The conjecture is expressed as follows\cite{LiuTian}.

\begin{theorem}(Arnold conjecture)

    Let $(M,\omega)$ be a closed symplectic manifold. A Hamiltonian diffeomorphism $\varphi :M\rightarrow M$ is called nondegenerate if its graph intersects the diagonal of $M\times M$ transversely. For nondegenerate Hamiltonian diffeomorphisms, the number of fixed points of $\varphi$ is greater than or equal to the sum of Betti numbers over any filed $\mathbb{F}$, namely
    \begin{center}
        $\sum_{i=0}^{2n}dim H_i(M;\mathbb{F})\leq$ \#\{fixed points of $\varphi$\}
   \end{center}
\end{theorem}

The nondegenerate case can also be expressed for any critical point $\gamma$ of $\mathcal{A}_H$, $det(Id-T_{\gamma(0)}\psi^1)\neq0$. But this case is not always satisfied, so there is also a weak version of Arnold conjecture, in which the lower bound is the number of critical points of any function. Here, we mainly consider a degenerate version, namely

\begin{theorem}(degenerate Arnold conjecture)

    Let $(M,\omega)$ be a closed symplectic manifold. For all Hamiltonian diffeomorphisms, the number of fixed points of $\varphi$ is greater than or equal to the cup length of cohomology ring of M, namely
    \begin{center}
        $cuplength\ of\ H^*(M)\leq$ \#\{fixed points of $\varphi$\}
   \end{center}    
\end{theorem}

To solve this problem, Conley-Zehnder first gave a proof for torus in 1983\cite{ConleyZehnder1983}. Floer established Floer homology and successfully solved Arnold conjecture when the symplectic manifold is monotone\cite{Floer1989}. This method is extended by Le an Ono to the negatively monotone case\cite{LeOno}. Ono did some results by himself in semipositive case\cite{Ono1995}, Hofer and Salamon also did some works in semipositive case\cite{HoferSalamon1995}. In 1999, Fukaya and Ono solved the problem for all symplectic in $\mathbb{Q}$ coefficient\cite{FukayaOno}. Then Mohammed Abouzaid and Andrew J. Blumberg proved it in $\mathbb{Z}/p\mathbb{Z}$ coefficient in 2021\cite{AbouBlum21}. And $\mathbb{Z}$ coefficient case was finally solved by Bai-Xu. So far, the nondegenerate Arnold conjecture is totally solved\cite{Xu-Bai}\cite{Rezchikov}. 

But the degenerate version is just partially solved. Floer first gave the proof of a special case, when the manifold is $\mathbb{CP}^n$, in his early paper\cite{Floer1989}. In this paper, he used the action of quantum cohomology on the Floer chain complex, which he called the “cap action”. And he proved in general case when $\omega(H_2(M)) = 0$ in \cite{FloerLag}. But quantum cohomology is not so well known at that time, so there are something not very clear in his paper. Following his idea, this paper will give an explicit proof of his result and consider the result in some other special manifolds. And we will introduce some other results at first.

\section{Some Development of Degenerate
Arnold Conjecture.}

\subsection{Result of degenerate version on toric manifolds}
In 1995, Givental published a paper in $The\ Floer\ Memorial\ Volume$ and gave the result\cite{Givental}. We mainly introduce his work in this section.

\begin{theorem}[](Givental)

    Let M be a compact toric maniflod with an integer class p(M) of the symplectic form. Then for any hamiltoonian diffeomorphism $h:M\rightarrow M$ 

    (i)the number of its fixed points is not less than 
\begin{center}
    $\max\limits_{\mathcal{E}} {(c(M),\gamma)/(p(M),\gamma)}$
\end{center}

    (ii)the total multiplicity of its fixed points is not less than dim $H^*(M,\mathbb{C})$
\end{theorem}
in which $c(M)$ is first chern class and $\mathcal{E}$ is the set of all non-zero effective homology classes, that is, Poincar$\acute{e}$-dual to some holomorphic hypersurfaces.

Toric manifold here means a symplectic quotient of $\mathbb{C}^n$ by a subtorus $T^k\subset T^n$. $T^n$ acts on $\mathbb{C}^n$ by $(a_1,a_2,\dots,a_n)\cdot(z_1,z_2,\dots,z_n) = (a_1z_1,a_2z_2,\dots,a_nz_n)$. 

Givental derives this fixed-point estimate from the following lattice version of the statement. The role of the vector $m$ is to measure how the equivariant cohomology classes shift under the torus action.
\begin{theorem}
    Let m = $(m_1,m_2,\dots,m_n)\in \mathbb{Z}_+^n\cap\mathbb{Z}^k$, $\mathbb{Z}^k$ here is induced by $ LieT^k \subset LieT^n$, be a nonnegative point in $\mathbb{Z}^k\subset\mathbb{Z}^n$, $m_i\geq 0$, $m\neq0$, $p\in\mathbb{Z}^{k*}$ be a primitive regular value of $\pi:\Pi = \{t\in \mathbb{R}^{n*}|t_i\geq0,i = 1,2,\dots,n\}\rightarrow\mathbb{R}^{k*}$. Then for any hamiltonian diffeomorphism $M\rightarrow M$, the number of its fixed points is not less than $(m_1+m_2+\dots +m_n)/\langle p,m\rangle$.
\end{theorem}

This result is very powerful because it will hold for all Hamiltonians without any assumptions. But the limitation lies in the fact that this estimate is weaker than cuplength estimate. Moreover, the class $p$ is highly dependent on the symplectic structure on M.

Then we introduce some notations used by Givental

$\mathbb{C}[u]:=\mathbb{C}[u_1,\dots,u_n]$, $\mathbb{C}[u,u^{-1}]:=\mathbb{C}[u_1,\dots,u_n,u_1^{-1},\dots,u_n^{-1}]$

I = the ideal in $\mathbb{C}[u] = \mathbb{C}[u_1,u_2,\dots,u_n]$ generated by linear equations of $\mathbb{R}^k = LieT^k \subset LieT^n = \mathbb{R}^n$

$J_r$ = $(u^m,m\in \mathbb{Z}^k|\langle p,m\rangle\geq r)_{\mathbb{C}[u]}$, here p is a primitive regular value of $\pi:\Pi \rightarrow\mathbb{R}^{k*}$ like above.

$\mathscr{J}_r$ means the image of $J_r$ in the quotient algebra $\mathscr{R} = \mathbb{C}[u,u^{-1}]/\mathbb{C}[u,u^{-1}]I$,

$\mathscr{J}$ is a $\mathbb{C}[u]$ submodule such that $\mathscr{J}_{r_+}\subset\mathscr{J}\subset\mathscr{J}_{r_-}$ for some $r_+>r_-$.

A proposition proved by Givental is that the quotient algebra $\mathbb{C}[u]/(I+\mathbb{C}[u]\cap J_r)$ is finite dimensional. A direct corollary of this proposition is there exists $q\in\mathcal{R}$ such that $q\notin \mathscr{J}$ but $u_1q,u_2q,\dots,u_nq\in\mathscr{J}$. We refer to \cite{Givental} for the proof. And the following sketch of proof also follows \cite{Givental}.

Following Fortune and Weinstein's idea\cite{BFAW}, we consider that any hamiltoian isotopy $h^t$ of a compact toric manifold $M = \mathbb{C}^n//T^k$ can be lifted up to  a $T^k$-equivariant homogeneous hamiltonian isotopy in $\mathbb{C}^n$(we also use notation $h^t$ to express this isotopy on $\mathbb{C}^n$).

Now let $\mathcal{H}^t$ denote the hamiltonian w.r.t $h^t$ on $\mathbb{C}^n\times [0,1]$, we define the action functional 
\begin{align*}
    &\mathcal{A}:\mathcal{L}\mathbb{C}^n\times\mathbb{R}^k\rightarrow\mathbb{R}\\
    &\mathcal{A} = \oint pdq - \oint \mathcal{H}^tdt-\lambda_1\oint \mathcal{P}_1dt-\dots-\lambda_k\oint \mathcal{P}_kdt
\end{align*}
$\mathcal{P} = (\mathcal{P}_1,\dots,\mathcal{P}_k)$ are components of the momentum map, $\lambda = (\lambda_1,\dots,\lambda_k)\in\mathbb{R}^k$ are Lagrange multipliers. Denote $\mathcal{A}|_{\mathcal{L}\mathbb{C}^n\times\{\lambda\}}$ by $\mathcal{A}_\lambda$, note that the reason we introduce Lagrangian multipliers here is the lift of a trajectory with respect to fixed point of $h^1$ on $M$ should not be closed under hamiltonian flow of $\mathcal{H}^t$. So we need to add some factors to $\mathcal{H}^t$ to ensure that it is lifted to a loop in $\mathbb{C}^n$. 

Let $\mathcal{S}$ denote the sphere of all rays in the loop space $\mathcal{L}\mathbb{C}^n\backslash$\{zero loop\}. Let $A = [\mathcal{A}^{-1}(0)\backslash (0\times \mathbb{R}^k)]/\mathbb{R}^\times_+\subset\mathcal{S}\times\mathbb{R}^k$, pull back $p:\mathbb{R}^k\rightarrow\mathbb{R}$ to $\mathcal{S}\times\mathbb{R}^k$ and restrict it on $A$, denote it by $\hat{p}:A\rightarrow\mathbb{R}$.

Then there is a one to one correspondence between fixed points of $h:M\rightarrow M$ and $\mathbb{Z}^k$-lattices of critical $T^k$-orbits of the function $\hat{p}$.

The generating function method is a central idea in Givental's proof of the symplectic fixed point theorem, which offers a significant simplification over traditional methods such as Floer homology. The method works by translating the problem of counting fixed points of a Hamiltonian diffeomorphism into a problem of finding critical points of a suitable generating function, which is defined on a finite-dimensional space. This allows the infinite-dimensional problem of fixed point counting to be approximated and ultimately solved in a more algebraically tractable manner.

The first key idea is to approximate the infinite-dimensional loop space of the symplectic manifold by a finite-dimensional space. That is, we can decompose $h^1$ into several hamiltonians with small Hofer energy, $h^1 = h_N\circ h_{N-1}\circ\dots\circ h_1$. Thus we can define two symplectic morphisms on $M^N = (M\times\dots\times M,\omega\oplus\dots\oplus\omega)$ to itself:
\begin{align*}
    &h^N = (h_1,h_2,\dots,h_N) \\
    &q^N:(x_1,x_2,\dots,x_N)\rightarrow(x_2,\dots,x_N,x_1).
\end{align*}
 Denoted by $\mathscr{H}$ and $\mathscr{Q}$ the generating function of $h^N$ and $q^N$. Let $\mathscr{F} = \mathscr{H}-\mathscr{Q}$, the fixed points of $h^1$ is one to one correspond to critical point of $\mathscr{F}$. To make sure the graph of $q^N$ is a Lagrangian subspace, we should replace $h^1$ by $-id\circ h^1$ and $q^N$ by $q^N(x_1,x_2,\dots,x_N) = (x_2,\dots,x_N,-x_1)$. Moreover, we will replace $N$ by even number $2N$.

This $\mathscr{F}$ is generating function of $\mathcal{A}_0$, we can find generationg function for any $\mathcal{A}_\lambda$, $\lambda\in$ the Lie algebra $\mathbb{R}^k$ by replacing $h_{2N}\circ h_{2N-1}\circ\dots\circ h_1$ by 
\begin{center}
    $h_{2N_2}\circ h_{2N_2-1}\circ\dots\circ h_1\circ \underbrace{t\circ t\circ\dots\circ t}_{2N_1times}$,
\end{center}
here $t = exp(\lambda/2N_1)$ is a sufficient small hamitonian diffeomorphism. Then we can get $\mathscr{F}_\lambda^N:(\mathbb{C}^n)^{2N}\rightarrow\mathbb{R}$. But $N = N_1+N_2$ is dependent on $\lambda$, so we can only use a uniform N for $\lambda$ in a compact cube $\Lambda_N$. Denote it by 
\begin{center}
    $\mathscr{F}_N:(\mathbb{C}^n)^{2N}\times\Lambda_N\rightarrow \mathbb{R}$, $\Lambda_N\subset\mathbb{R}^k$, $\cup\Lambda_N = \mathbb{R}^k$.
\end{center}
Now let $\mathcal{S}_N$ be the unit sphere in $\mathbb{C}^{2nN}$, we can split it into two parts by value of $\mathscr{F}_N$
\begin{center}
    $F_N^{\pm} = \{(x,\lambda)\in\mathcal{S}_N\times\Lambda_N|\mathscr{F}_\lambda^N(x)\geq0(resp.\ \leq 0)\}$.
\end{center}
and their intersection is 
\begin{center}
    $F_N^0 = (\mathscr{F}_N^{-1}(0)\backslash\{0\}\times\Lambda_N)/\mathbb{R}_+^\times = F_N^+\cap F_N^-$.
\end{center}
And denote $\hat{p}_N:F_N^0\subset\mathcal{S}_N\times\Lambda_N\rightarrow\Lambda_N\subset\mathbb{R}^k\xrightarrow{p}\mathbb{R}$, critical $T^k$ orbits correspond(not one to one) to fixed points of $-id\circ h^1$ on $M$. If we denote $\mathscr{G}_m^K:\mathbb{C}^{2nK}\rightarrow\mathbb{R}$ the generating function of $-id\circ\underbrace{t\circ\dots\circ t}_{2Ktimes}$, $t = exp(m/2K)\in T^k$, $m\in \mathbb{Z}^k$ and $\Lambda_N+m\subset\Lambda_{N+K}$. A natural result we can expect is that $\mathscr{F}_{N+K}|_{\Lambda_N+m}$ is fiberwise homotopy equivalent to the fiberwise suspension $\mathscr{F}_N\oplus_{\Lambda_N}\mathscr{G}_m^K$. 

Then we consider equivariant cohomology defined by $H^*_G(X) = H^*(X_G = X\times EG/G)$ where $X$ has a G-action on it. Let $\partial F_N^\pm = F_N^\pm\cap (\mathcal{S}_N\times\partial\Lambda_N)$ and we will deal with $H^*_{T^k}(F_N^-,\partial F_N^-)$. In the following part, we do not perform any computations. We refer more details of computations to \cite{Givental}, and we just introduce the consequences here.

Let $\hat{\mathscr{F}}_{N+K} = \mathscr{F}_N\oplus_{\Lambda_N}\mathscr{G}_m^K$, it was computed that
\begin{align}\label{formula1}
    \tilde{H}^*_{T^k}(\hat{F}_{N+K}^-,\partial \hat{F}_{N+K}^-)& = \tilde{H}^*_{T^k}(F_N^-,\partial F_N^-)\otimes I^{K+m},\nonumber
    \\
    J_{\hat{F}} & = J_F\otimes I^{K+m},
\end{align}
$I^{K+m}$ is the principal ideal in $H^*_{T^k}(pt)$ generated by the monomial $u^{K+m} = u_1^{K+m_1}\cdot u_2^{K+m_2}\cdot{} \dots{} \cdot u_n^{K+m_n}$. Note that $H^*_{T^k}(pt)$ is a quotient of $H^*_{T^n}(pt) = \mathbb{C}[u_1,u_2,\dots,u_n]$.

With these propositions, we can form an asymptotic computation of $H^*_{T^k}(F^-,\partial F^-)$ generated by $\mathscr{F}_N$ because they form a direct system. This computation is the key idea of Givental's proof, it finds a method to describe the fixed points using algebraic structure. This is also the reason why it is hard to extend to other manifolds, namely, this algebraic structure is very special and can not be realized on any symplectic manifolds. 

Restrict $\mathscr{F}_N$ on a convex compact r-dimensional polyhedron $\Gamma\subset \mathbb{R}^n$ and 
\begin{align*}
    \mathscr{H}^*_{T^n}(\Gamma) &= \lim\limits_{N\rightarrow\infty}u^{-N}H^{*+2N}_{T^n}(F^-_N|_\Gamma,\partial F^-_N|_\Gamma)\\
    & = H^r(\Gamma,\partial\Gamma)\otimes_{\mathbb{C}}\mathbb{C}[u,u^{-1}]/J_{\Gamma}, 
\end{align*}
where $J_\Gamma$ is the $\mathbb{C}[u]$ submodule in $\mathbb{C}[u,u^{-1}]$ generated by $u^l$, $l\in\Delta_\Gamma = \Gamma+\mathbb{R}_+^n = \{\mu\in\mathbb{R}^n|\exists\gamma\in\Gamma,\ \mu_i>\gamma_i,\ i = 1,2,\dots,n \}$. 

Choose $\Gamma_N(\nu) = \Lambda_N\cap p^{-1}(\nu)$ for some fixed $\nu\in\mathbb{R}$, denote by $\mathscr{H}^*_{T^{k}}(\Gamma(\nu))$ the limit
\begin{center}
    $\lim\limits_{N\rightarrow\infty}\mathscr{H}^*_{T^{k}}(\Gamma_N(\nu)) = H^{k-1}(\Gamma(\nu),\infty)\otimes_{\mathbb{C}}\mathscr{R}/\mathscr{J}_{\Gamma(\nu)}$.
\end{center}
and denote $\mathscr{H}^*_{T^k}(F^-(\nu))$ the direct limit of 
\begin{center}
    $H^{*+2N}_{T^k}(F^-_N(\nu),\partial F^-_N(\nu))\rightarrow H^{*+2N'}_{T^k}(F^-_{N'}(\nu)|_{\Gamma_N(\nu)},F^-_{N'}(\nu)|_{\partial\Gamma_N(\nu)})$.
\end{center}
There are natural maps 
\begin{center}
    $H^{*+2N}_{T^k}(\Gamma_N,\partial\Gamma_N)\rightarrow H^{*+2N}_{T^k}(F^-_N(\nu),\partial F^-_N(\nu))$,
\end{center}
so there is a map between these two direct limits
\begin{center}
    $\mathscr{H}^*_{T^{k}}(\Gamma(\nu))\rightarrow \mathscr{H}^*_{T^k}(F^-(\nu))$.
\end{center}
We denote its kernel by $\mathscr{J}^*(F^-(\nu))$. Because $\mathscr{F}_{N+K}|_{\Lambda_N+m}$ is fiberwise homotopy equivalent to the fiberwise suspension $\mathscr{F}_N\oplus_{\Lambda_N}\mathscr{G}_m^K$, together with \ref{formula1},  there is an isomorphism 
\begin{center}
    $\mathcal{U}_m:\mathscr{H}^*_{T^k}(F^-(\nu))\rightarrow\mathscr{H}^*_{T^k}(F^-(\nu+p(m)))$.
\end{center}
Therefore it induces an isomorphsim on $\mathscr{J}^*(F^-(\nu))$. 

We should remark here for general generating function, we decompose $h$ into $h_1\circ\dots\circ h_N$, and if we compute the Floer cohomology $HF^*(H,Q)$, it has the same critical points as $HF^*(G_h,\Delta)$ where $G_h$ is the graph of $h$ in $M\times M$ and $H$, $Q$ means the graph of $h^N$, $q^N$ above. Although the ring structure of them is not obviously the same, it was believed by author. And Givental use the generating function with a Lie group action, which is also believed that it could be translate to the language of Floer homology.

The isomorphism $\mathcal{U}_m$ is just multiplication of $u^m$ on $\mathscr{R}$, that means $\mathscr{J}^*(F^-(\nu))$ is invariant under this multiplication if $p(m) = 0$. One only need one more proposition to finish the whole proof, that is

(1) If there is no critical values of $\hat{p}$ in $[\nu_0,\nu_1]$, then $\mathscr{H}^*_{T^k}(F^-(\nu_0))\cong\mathscr{H}^*_{T^k}(F^-(\nu_1))$.

(2) If there is only one critical value of $\hat{p}$ in $[\nu_0,\nu_1]$ and all critical $T^k$-orbits on this level are isolated. Let $v\in H^*_{T^k}(pt)$ be a positive degree element and $q_0\in\mathscr{H}^*_{T^k}(F^-(\nu_0))$, $q_1\in\mathscr{H}^*_{T^k}(F^-(\nu_1))$ be images of same $q\in\mathscr{R}$ under $\mathscr{H}^*_{T^{k}}(\Gamma(\nu))\rightarrow \mathscr{H}^*_{T^k}(F^-(\nu))$, then $q_0 = 0$ implies $vq_1 = 0$.

Now, we are going to count critical values of $\hat{p}:A\rightarrow\mathbb{R}$ between two non-singular values $\nu$ and $\nu+1$. 

Use the corollary above to $\mathscr{J} = \mathscr{J}^*(F^-(\nu))$, we can find some $q\in\mathscr{R}$, $q\notin \mathscr{J}$ but $u_1q,u_2q,\dots,u_nq\in\mathscr{J}$. If the total number of fixed points less than $m_1+\dots+ m_n$, by (2) we get $u_1^{m_1}u_2^{m_2}\cdots u_n^{m_n}q$ is 0 in $\mathscr{H}^*_{T^k}(F^-(\nu+\langle p,m\rangle))$, thus lies in $\mathscr{J}^*(F^-(\nu+\langle p,m\rangle))$. On the other hand, this multiplication is an isomorphism of $\mathscr{H}^*_{T^k}(F^-(\nu))$ and $\mathscr{H}^*_{T^k}(F^-(\nu+\langle p,m\rangle))$, but $q\notin \mathscr{J}^*_{T^k}(F^-(\nu))$, this contradiction completes the proof.

This work contains Floer's result about projective manifolds. As an example, for $\prod\mathbb{CP}^{n_i}$, Givental's work can give result $\max \{n_i\}$, which is much stronger than the result $gcd(n_i)$ given by Floer.  

Notice that the value $\max\limits_{\mathcal{E}}(c(M),\gamma)/(p(M),\gamma)$ is not equal to cuplength. For example, If  $M = S^2\times S^2$, the cuplength is 2, so cuplength estimate said it will have 3 fixed points but $\max\limits_{\mathcal{E}}(c(M),\gamma)/(p(M),\gamma)$ is at most 2. If the symplectic area of two factors has fractional ratio, the result is 1, and it can give nothing if this ratio is irrational.

In example $M = $ 1 point blow-up of $\mathbb{CP}^n$, it is similar to the above case. Namely, the result is less than cuplength if the ratio of two factors is fractional and get nothing if the ratio is irrational.

There is a very natural observation, this generating function method can be repeated for any symplectic quotient $Y = X//G$, but the obstruction here is to find how the group $H_G^*(pt)$ acts on
$\mathscr{H}^*_{G}(F^-(\nu))$, there should be something interesting but things will be more complicate if we consider general objects. We can guess that if $H_G^*(pt) = H^*(G)$ has cuplength $k$, then $H$ be any Hamiltonians on quotient type symplectic manifold $Y = X//G$, $\psi_H^1$ will have at least $k+1$ fixed points.

\subsection{An estimate theorem dependent on Hofer energy}
In 1997, inspired by Givental and Floer's result, Matthias Schwarz introduced the action of quantum-cohomology on the Floer chain complex and proved the case when the hamiltonian  is not too big in his work\cite{Schwarz}. The following sketch of proof is mainly follow this work. The concept of quantum cohomology and quantum cup length will be introduced in Chapter 3.

\begin{theorem}\label{Schwarzthm}(Schwarz)

    Let $\psi$ be a Hamiltonian automorphism of a closed symplectic manifold (M,$\omega$) satisfying the rationality condition $\omega(\pi_2(M))=p\cdot \mathbb{Z},p>0$. If $d(\psi,id)\leq p$ then 
    \begin{center}
        $\# Fix\ \psi\geq cl(M)$
    \end{center}
    $d(\psi,id) := \inf\{||H||\ |\psi\ is\ generated\ by\ H\}$ and $cl(M)$ is cup-length of M. Moreover if M is weakly monotone(i.e. semi-positive)
    then for every $\psi$ we have
    \begin{center}
        $\# Fix\ \psi\geq\sup\limits_{A\in \pi_2(M)}\frac{p\ qcl(A)}{p+||\psi||+\omega(A)}$
    \end{center}
    where qcl(A) is quantum cup-length w.r.t quantum cup $*_A$, $||\psi|| = \inf\{||H||_1|\ \psi = \psi_H^1\}$.
\end{theorem}
It's easy to find that this estimate is dependent on the energy of H, although it will work without any assumption about the symplectic manifold itself.

 Schwarz use an alternative version of quantum cohomology induced by Morse theory. That is, for a given Morse function $f$ on $M$, define the chain complex by
\begin{center}
    $QC^k(f, \mathbb{Z}_2) = \{ \varphi \in \mathbb{Z}_2^{(\text{Crit } f \times \Gamma)_k} \mid$ \\ $
    \ \#\{(x, A) \mid \varphi(x, A) \neq 0,\, \omega(A) \leq r\} < \infty \text{ for all } r \in \mathbb{R} \}$,\\
    $(\text{Crit } f \times \Gamma)_k = \{(x, A) \in \text{Crit } f \times \Gamma \mid \mu(x) + 2c_1(A) = k\} $,
\end{center}
$\Gamma$ here means the group $H_2(M)/(\ker\omega\cap\ker c_1)$. We have $QC^*(f, \mathbb{Z}_2) = C^*(f,\mathbb{Z}_2)\otimes_{\mathbb{Z}_2}\Lambda$ as graded rings, $\Lambda$ is the Novikov ring 
\begin{center}
    $\Lambda(\Gamma, \mathbb{Z}_2) = \left\{ \lambda \in \mathbb{Z}_2^\Gamma \middle| \#\{A \in \Gamma \mid \lambda(A) \neq 0, \omega(A) \leq r\} < \infty \text{ f.a. } r \in \mathbb{R} \right\}$.
\end{center}
The coboundary operator is defined by $\delta\otimes id$ where $\delta$ is the coboundary operator of $C^*(f,\mathbb{Z}^2)$.

The deformed cup product on it is defined in the following way. $(f,g_0)$ is the Morse pair on M, let $(g_s)_{s\in\mathbb{R}}$ be a family of Riemannian metrics on M s.t. $g_s = g_0,|s|\geq 1$. For any $x\in Crit(f)$
\begin{center}
    $W^u(x) = \left\{ \gamma \in C^\infty((-\infty, 0], M) \,\middle|\, \dot{\gamma}(s) + \nabla_{g_s}f(\gamma(s)) = 0,\, \gamma(-\infty) = x \right\}, $  \\  
    $W^s(x) = \left\{ \gamma \in C^\infty([0, \infty), M) \,\middle|\, \dot{\gamma}(s) + \nabla_{g_s}f(\gamma(s)) = 0,\, \gamma(+\infty) = x \right\}.$
\end{center}
Let $\mathcal{M}^A_3(J)$ be the moduli space of $J-$holomorphic  spheres representing class $A$ with three marked points $0,1,\infty$. Given $x,y,z\in M$, we define $\mathcal{M}^A_{z;x,y}(J,f,g^1,g^2,g^3) = \mathcal{M}^A_3(J)\times_{ev}(W^s(x)\times W^s(y)\times W^u(z))$, $ev(u) = (u(0),u(1),u(\infty))$ $ev(\gamma^1,\gamma^1,\gamma^3) = (\gamma^1(0),\gamma^2(0),\gamma^3(0))$. Note that for every un-/stable manifold, we allow it to use different metric families. 

If $M$ is weakly monotone, then $\mathcal{M}^A_{z;x,y}(J,f,g^1,g^2,g^3)$ is compact in dimension zero for generic $J$ and $(g^i)_s$. Thus we can define 
\begin{center}
    $n_A(z; x, y) = \#\mathcal{M}_{z,x,y}^A \pmod{2} \quad \text{for } \mu(z) + 2c_1( A )= \mu(x) + \mu(y)$\\
    $m_A : C^k(f, \mathbb{Z}_2) \otimes C^l(f, \mathbb{Z}_2) \to C^{k+l-\deg A}(f, \mathbb{Z}_2)$ \\
    $m_A(x, y) = \sum_{z \in \text{Crit}_{k+l-\deg A}} n_A(z; x, y) z \,.$
\end{center}
The operator $m_A$ is commute with the coboundary operator, so it will induce an operator on cohomology level, we denote it by $*_A$. In fact, this is the same operator as what we defined in chapter 2. And we can define the quantum cup product  
\begin{center}
    $a*b = \sum\limits_{A\in \Gamma}a*_Ab\cdot e^A$
\end{center}
on $H^*(M)$ and extend it linearly onto $QH^*(M)$. Similarly, we can define 
\begin{center}
    $\mathcal{M}^A_{x_0;x_1,\dots,x_k}(J,f,(g^i)_s) = \mathcal{M}^A_{k+1}(J)\times_{ev}(W^u_{g^0}(x_0)\times W^s_{g^1}(x_1)\times\dots\times W^s_{g^k}(x_k))$.
\end{center}
Elements in this moduli space looks like figure\ref{fig:elment1}

\begin{figure}[t]
\begin{center}
\includegraphics[width=0.6\textwidth]{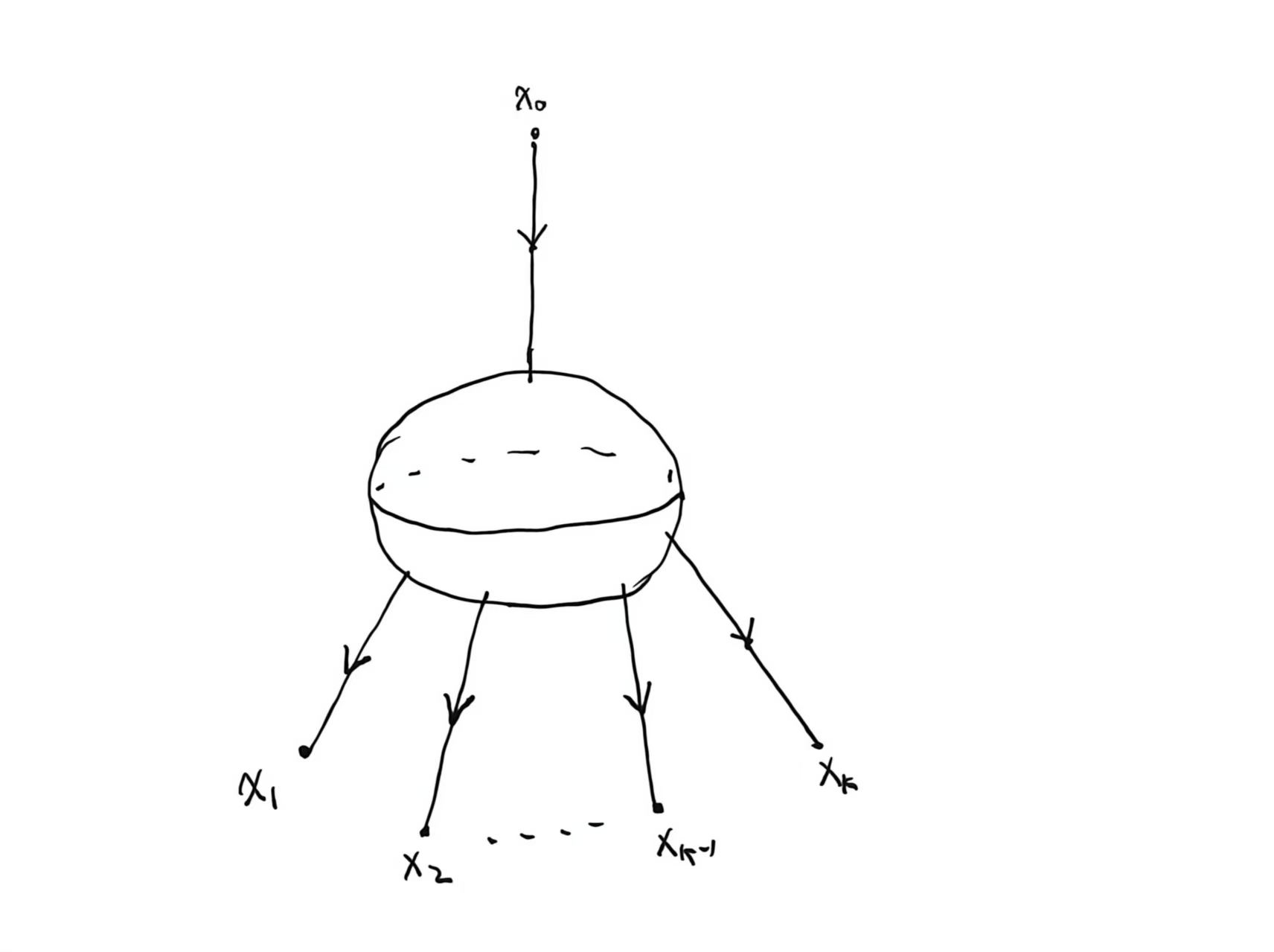}
\caption{Elements in $\mathcal{M}^A_{x_0;x_1,\dots,x_k}(J,f,(g^i)_s)$.}
\label{fig:elment1}
\end{center}
\end{figure}

It will induce the operator
\begin{center}
    $m_A : C^{i_1}(f) \otimes \dots \otimes C^{i_k}(f) \to C^{i_1 + \dots + i_k - \deg A}(f) $, \\
    $m_A(x_1 \otimes \dots \otimes x_k) = \sum\limits_{x_0} (\# \mathcal{M}_{x_0;x_1,\dots,x_k}^A(J, f, (g^i)_s)\ \mod2)  \ x_0 $.
\end{center}
Of course it will induce an operator on the cohomology level with the property
\begin{center}
    $\alpha_1 * \ldots * \alpha_k = \sum\limits_{A \in \Gamma} \mathfrak{m}_A(\alpha_1 \otimes \ldots \otimes \alpha_k) e^A$.
\end{center}

The key idea of Schwarz is to find a sequence of nontrivial elements in this formula and estimate the number of periodic loops on the limit broken trajectory.To realize this goal, one should make this moduli space easier to deal with, so an alternative moduli space is introduced here.
\begin{center}
    $\mathcal{M}_A(R) = \{ u \in C^\infty(\mathbb{R} \times S^1, M) \mid \overline{\partial}_R(u) := 0, \, E_R(u) < \infty, \, [u] = A \}$,\\
    $(\overline{\partial}_R u)(s,t) = \partial_s u + J_R(s,t,u) \left( \partial_t u - X_{H_R}(s,t,u) \right) \in C^\infty(u^* TM)$,\\
    $E_R(u) = \int\limits_{-\infty}^{\infty}\int_{S^1}|\partial_su|^2_{J_R}dsdt$.
\end{center}
Let $\beta:\mathbb{R}\rightarrow [0,1]$ a smooth monotone cutoff function, $\beta(s) = 
\begin{cases} 
0, & s \le 0 \\
1, & s \ge 1 
\end{cases}$.
Here $(J_R,H_R)$ is defined by $(J_R, H_R)(s, t, p) = $
\begin{center}
    $\begin{cases} 
    (J_0(p), 0), & s \le 0, \\
    (\overline{J}(s, t, p), \beta(s) H(t, p)), & 0 < s \le R, \\
    (\overline{J}(R + 1 - s, t, p), \beta(R + 1 - s) H(t, p)), & R < s \le R + 1, \\
    (J_0(p), 0), & R + 1 < s.
\end{cases}$
\end{center}
Where $\overline{J}(s,t,p) = \begin{cases}
    J_0, s\leq 0\\
    J, s\geq 1
\end{cases}$ for some fixed time-independent $J_0$ and our original $J$. Because we can regard $\mathbb{R}\times S^1$ as $S^2\backslash\{+\infty,-\infty\}$, and $u\in\mathcal{M}_A(R)$ has finite energy, we can use removal of singularity to extend u onto $S^2$, that means $[u]$ is a well-defined class in $H_2(M)$. Using this space instead of $\mathcal{M}^A_{k+1}(J)$, we get
\begin{center}
    $\mathcal{M}^A_{x_0;x_1,\dots,x_k}(J,f,(g^i)_s) = \mathcal{M}_A((k+1)R)\times_{ev}(W^u_{g^0}(x_0)\times\dots\times W^s_{g^k}(x_k))$.
\end{center}
Evolution map on $\mathcal{M}_A((k+1)R)$ is $ev(u) = (u(-\infty),u(R,0),\dots ,u(kR,0))$. The moduli space and the operator $m_A$ defined by this way coincide with the above method. Elements in this new moduli space have the form in figure\ref{fig:elment2}

\begin{figure}[t]
\begin{center}
\includegraphics[width=0.8\textwidth]{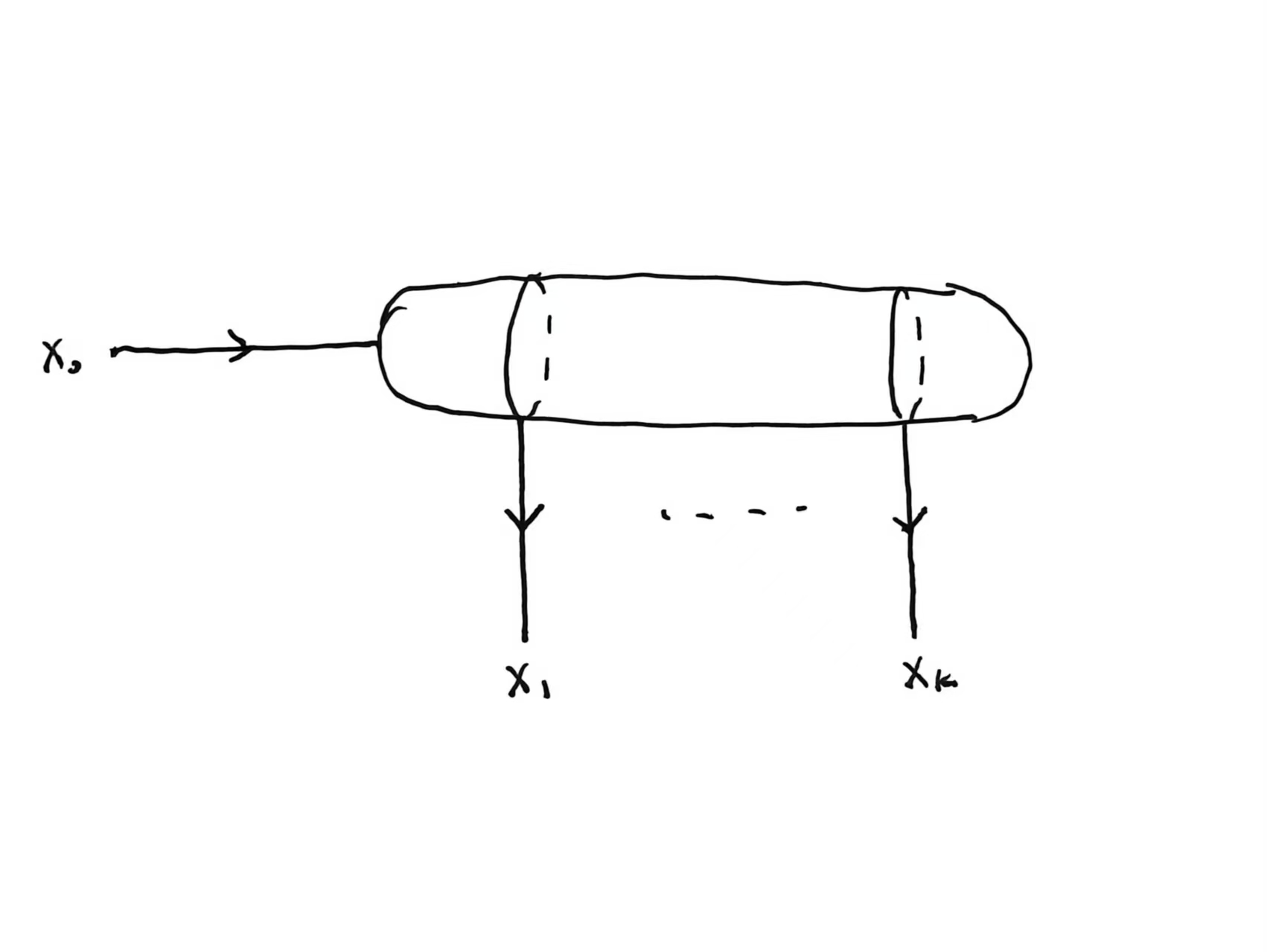}
\caption{Elements in the new moduli space.}
\label{fig:elment2}
\end{center}
\end{figure}

Let $\tilde{\mathcal{P}}(H)$ be the covering of the set of 1-periodic loops of $\psi_H^1$ such that $\mathcal{A}_H$ is well defined on it. For any $x,x'\in\tilde{\mathcal{P}}(H)$, denote by $\mathcal{M}_{x,x'}(J,H)$ the set of $u:\mathbb{R}\times S^1\rightarrow M$, $\bar{\partial}_{J,H}u = 0$ and $x = (\gamma,v) \ \ x' = (\gamma',v')$, $\lim\limits_{s\rightarrow -\infty}u(s) = \gamma$ $\lim\limits_{s\rightarrow \infty}u(s) = \gamma'$ $v\#u = v'$. We say $x\leq x'$ if  $\bar{\mathcal{M}}_{x,x'}(J,H)$ is nonempty, and $x<<x'$ if $\exists u\in \mathcal{M}_{x,x'}(J,H)$ and $y\in Crit(f),\ \mu(y)\geq1$ such that $u(0,0)\in W^s_g(y)$. Here,  $\bar{\mathcal{M}}_{x,x'}(J,H)$ is the compactification of $\mathcal{M}_{x,x'}(J,H)$, which consists of broken trajectories.

If $qcl(A) = k+1$, $\mathcal{M}^A_{x_0;x_1,\dots,x_k}((k+1)R_n)$ with $x_i\in Crit(f)$ $\mu(x_i)\geq1$ is nonempty by definition. So we can choose $u_n\in\mathcal{M}^A_{x_0;x_1,\dots,x_k}((k+1)R_n)$, $R_n\rightarrow\infty$ as $n\rightarrow \infty$. Moreover, we choose $u_n$ such that $\dim_{loc}u_n = 0$, so $|\nabla u_n|$ can be uniformly bounded. Otherwise, there will be some bubble case in the limit, but it is impossible by dimension reason. We can choose a subsequence weakly converges to a broken trajectory $(v_-,v_1,\dots,v_N,v_+)\in \mathcal{M}_{x_0}^-(\bar{J}, H) \times \mathcal{M}_{x_0, x_1}(J, H) \times \ldots \times \mathcal{M}_{x_{N-1}, x_N}(J, H) \times \mathcal{M}_{x_N}^+(\bar{J}, H)$, where
\begin{center}
    $\mathcal{M}_x^\mp(\bar{J}, H) = \{ u : \mathbb{R} \times S^1 \to M |$ \\
    $ \partial_s u + \bar{J}(\pm s, t, u)(\partial_t u - \beta(\pm s) X_H(t, u)) = 0,\, E(u) < \infty,\, u(\pm\infty) = x \}$.
\end{center}

In this limit process, a transverse relation is required
\begin{center}
    $\bigcup\limits_{\mu(u)\geq1}W^s_g(y)\cap\text{Fix}\ \psi^1_H = \phi$.
\end{center}
If $\tilde{P}(H)$ is finite, then for generic $(f,g)$ this condition is satisfied with respect to all $H_n$. 

We can take a reparametrization sequence $(\sigma_{i,n})$ for $i = 1,2,\dots,N$ so that $u(\cdot + \sigma_{i,n})\rightarrow v_i$ in $C^\infty_{loc}$ as $n\rightarrow\infty$. Choose a suitable subsequence such that $u(\cdot-jR_n,\cdot)\rightarrow w_j \in \mathcal{M}_{x_j,x_j'}(J,H)$ for some $x_j,x_j'\in\tilde{\mathcal{P}}(H)$. By definition, we will get $x_1<<x_1'\leq x_2<< x_2'\leq\cdots\leq x_k<<x_k'$, $w_j$'s are pairwisely different because $R_n\rightarrow\infty$. Then 
\begin{center}
    $\mathcal{A}_H(x_1)<\dots<\mathcal{A}_H(x_k)<\mathcal{A}_H(x_k')$.
\end{center}
After that, we can do an energy estimate, for any $u\in \mathcal{M}_A(R)$
\begin{align*}
    E(u)& = \int_{S^2}\omega(\partial_su,J_R\partial_su)\\
    & = \int_{S^2}\omega(\partial_su,\partial_tu-X_{H_R})\\
    & = \int_{S^2}u^*\omega+\int_{S^2}\frac{\partial H_R}{\partial s}dsdt\\
    & = \omega(A)+\int_{\mathbb{R}}\chi'\cdot\int_{S^1}H(t,u(s,t))dtds.
\end{align*}
$H_R = \chi H$ where $\chi = \begin{cases}
    0, & s \le 0, \\
    \beta(s), & 0 < s \le R, \\
    \beta(R + 1 - s), & R < s \le R + 1, \\
    0, & R + 1 < s.
\end{cases}$
So $E(u)\le u^*\omega +\int_{S^1}\sup H-\int_{S^1}\inf H = \omega(u)+\|H\|$. Using this estimate, together with $(v_-,v_1,\dots,v_N,v_+)$ is a limit of elements in $\mathcal{M}_A(R_n)$, we can find some $l\in\mathbb{R}$ such that 
\begin{center}
    $l<\mathcal{A}_H(x_1)<\dots<\mathcal{A}_H(x_k)<\mathcal{A}_H(x_k')<l+\omega(A)+\|H\|$.
\end{center}

Let $h_{eff}(\omega,H,J)$ be the minimal energy of closed broken trajectories in some $\bar{\mathcal{M}}_{x,x}$. Here, $\mathcal{M}(x,y)$ is the moduli space of 
\begin{center}
    $\{u:\mathbb{R}_\tau\times S^1_t\rightarrow M|\bar\partial_Hu = \partial_\tau u+ J\partial_t u +X_{H_t}(u(\tau,t)) = 0,\lim\limits_{\tau\rightarrow-\infty}u(\tau,\cdot)= x,\lim\limits_{\tau\rightarrow+\infty}u(\tau,\cdot)= y\}$.
\end{center}
and $\bar{\mathcal{M}}_{x,y}$ is its compactification. The energy $h_{eff}(\omega,H,J)$ is something like the minimal energy of holomorphic spheres in a closed symplectic manifold and it is also positive by Gromov compactness. Under the condition $\omega(\pi_2(M))=p\cdot \mathbb{Z},p>0$, a direct observation is
$h_{eff}(\omega,H,J)\geq p$.

By the definition, we have at least $(k+1)\max\{\frac{1}{n}|\frac{1}{n}<\frac{h_{eff}(\omega,H,J)}{\|H\|+\omega(A)}\})$ 1-periodic orbits. In case $\omega(H_2(M)) = p(\omega)\cdot\mathbb{Z}$ for some $p(\omega)\in\mathbb{Q}$, \ref{Schwarzthm} follows from this result by the estimate $h_{eff}(\omega,H,J)\geq p(\omega)$. As a corollary, if we take $A = 0$ and $\|H\|\leq h_{eff}(\omega,H,J)$, then we have at least $qcl(0) = cl(M)$ 1-periodic orbits.

The result of quantum cup length is similar to Givental's result above. He use the upper energy estimate of a family of trajectories, which use energy of Hamiltonian, to get the estimate of periodic orbits. So the cuplength estimate only holds for such Hamiltonians with not too big energy.

Moreover, there are also some recent results by Weimin Gong\cite{WeiminGong} we do not introduce here.

\section{The Quantum Cohomology Action on Floer Cohomology}

\subsection{quantum cohomology}

For a symplectic manifold $(M,\omega)$, the Novikov ring $\Lambda$ is formed by $\lambda = \sum\limits_{A\in H_2(M)}\lambda_Ae^A$ with each $\lambda_A$ lies in the coefficient ring $R$. And $e^A$ are formal variables subject to the relation $e^A*e^B = e^{A+B}$. Moreover, for each const C only finitely many $A$ with $\omega(A)\leq C$ have nonzero coefficient $\lambda_A$. Moreover, an equivalence relation defined on this set is $e^A\sim e^B$ if $c_1(A) = c_1(B)$, $\omega(A) = \omega(B)$.

The variable $e^A$ is considered to be of degree $2c_1(A)$.

\begin{definition}
    Let $H^*(M) = H^*(M,\mathbb{Z})/Torsion$, then the quantum cohomology is defined by
\begin{center}
    $HQ^*(M) = H^*(M)\bigotimes \Lambda$.
\end{center}
Its elements have form $\sum a_i\otimes \lambda_i$ and it is a graded ring with $deg(a_i\otimes\lambda_i) = \deg(a_i)+deg(\lambda_i)$. The ring structure is not induced by tensor, but defined by
\begin{center}
    $a*b = \sum\limits_{A\in \Gamma}a*_Ab\cdot e^A$
\end{center}
$a*_Ab$ is the form uniquely determined by $\int_M (a*_Ab)\cup c = GW_{0,3}^{M,A}(a,b,c)$. 
\end{definition}
$GW_{0,3}^{M,A}(a,b,c)$ is the Gromov-Witten invariant defined by
\begin{center}
    $GW_{0,3}^{M,A}(a,b,c) = \int_{\mathcal{M}_{0,3}(A)}ev_0^*a\wedge ev_1^*b\wedge ev_2^*c$.
\end{center}
$\mathcal{M}_{0,3}(A)$ is the moduli space of holomorphic Riemannian surfaces with genus 0 and 3 marked points. Note that to make this definition meaningful, the degree equation $\deg a+\deg b+\deg c = \dim \mathcal{M}_{0,3}(A) = 2n+2c_1(A)$ should be required. 

The ring $\{HQ^*(M),*\}$ is called small quantum cohomology, this ring structure is also denoted by $\cup^Q$.

\begin{definition}
    The quantum cuplength associated to $A\in H_2(M)$ is defined by \\
    \begin{center}
        $qcl(A) = \max \{k+1|the\ coefficient\ of\ e^A\ in\ \alpha_1*\alpha_2*\dots*\alpha_k\ is\ nonzero\ for\ some\ \alpha_i\in H^{d_i}(M),\ d_i\geq 1\}$.
    \end{center}
\end{definition}
Note that $\alpha_1*\alpha_2*\dots*\alpha_k = \alpha_1\cup\dots\cup\alpha_k+\sum\limits_{A\neq0}a_Ae^A$, so $qcl(0)$ is the usual cuplength of M.

Then we will introduce the action of this ring on Floer chain complex $CF^*(M,H) = \bigoplus\limits_{\gamma\in \mathcal{P}(H)}\Lambda \gamma$. $\mathcal{P}(H)$ is the set of all contractable 1-periodic orbits of $\psi$.

An element in $CF^*(M,)$ can be expressed by pair $(\gamma,u)$ where $\gamma\in\mathcal{P}(H)$ and $u$ is a disk filling of $\gamma$, namely, $u:\mathbb{D}^2\rightarrow M$, $\partial u = \gamma$. And $(\gamma,u)\sim (\gamma,\tilde{u})$ iff $\int_{\mathbb{D}^2}u^*\omega = \int_{\mathbb{D}^2}\tilde{u}^*\omega$, $\int_{\mathbb{D}^2}u^*c_1 = \int_{\mathbb{D}^2}\tilde{u}^*c_1$.

\begin{definition}

The definition of the action is

\begin{center}
    $\cup:HQ^*(M)\otimes CF^*(M)\rightarrow CF^*(M)$\\
    $\alpha\cup(\gamma,u) = \sum\limits_{\substack{\gamma'\in\mathcal{P}(H)\\u_i\in \mathcal{M}(\gamma,\gamma')\\u_i(0,0)\in |\alpha|}}(\gamma',u\#u_i)\ for\ \alpha\in H^*(M)$\\
    $e^A\cup(\gamma,u) = (\gamma,u\#A)$
\end{center}
$u\#u_1$ means the disk filling of $\gamma'$ given by gluing the image of $u$ and $u_1$. $|\alpha|$ means a representation the cycle of Poincar\'e duality of $\alpha$. In fact, this depends on the choice of the representation, but as Floer explained in \cite{Floer1989}, this action is commutative with the $\delta$ operator of $CF*(M)$, and $\alpha\cup$ induced by different representations are chain homotopic. So this definition will induce a well-defined action in homotopy level

\begin{center}
    $\cup:HQ^*(M)\otimes HF^*(M)\rightarrow HF^*(M)$
\end{center}

\end{definition}

Note that if we use Conley-Zehnder index $\mu$ on $CF^*(M)$, then the $\cup$ operator is graded. Then a natural proposition is 
\begin{center}
    $\alpha\cup(\beta\cup(\gamma,u)) = (\alpha\cup^Q\beta)\cup(\gamma,u)$
\end{center}
and it was proved in \cite{Schwarz}. 

\subsection{proof of $\mathbb{CP}^n$ case}

Now we give a proof of Floer's original result here. Notice that this result is already covered by Givenatl's result, we just write down an alternative method use more Floer theory here.

\begin{proposition}(Floer)
    For any Hamiltonian H on $\mathbb{CP}^n$, $\psi_H^1$ has at least $n+1$ fixed points.
\end{proposition}

\begin{proof}
    For any Hamiltonian H, we can perturbe it a little to $\tilde{H}$ such that $\psi_{\tilde{H}}$ is nondegenerate. Moreover, we can perturbe H only in a small neighborhood of 1-periodic orbits of $\psi_H$ so that each 1-periodic orbit $\gamma_i$ of $\psi_H$ will split to several  1-periodic orbits $\gamma_{i,j}$ of $\psi_{\tilde{H}}$ which lie in a small neighborhood of $\gamma_i$. We call $\{\gamma_{i,\cdot}\}$ a group of 1-periodic orbits for each i.

    We know the quantum cohomology ring of $\mathbb{C}P^n$ is $H^*(\mathbb{C}P^n)\otimes\Lambda(\mathbb{C}P^n) = H^*(\mathbb{C}P^n,\Lambda)$ with relation $a^{n+1} = e^A$, $a$ is a generator of $H^2(\mathbb{C}P^n)$. Then use the action of $HQ^*(\mathbb{C}P^n)$ on $HF^*(\mathbb{C}P^n,\tilde{H})$

    For any fixed element $(\gamma,u)$ in $CF^*(M)$, consider 
    \begin{center}
        $(a\cup)^{n+1}(\gamma,u) = (\cup^Qa)^{n+1}\cup(\gamma,u)+(\Phi\partial+\partial\Phi)(\gamma,u)$\\$ = e^A\cup(\gamma,u)+(\Phi\partial+\partial\Phi)(\gamma,u)$\\$ = (\gamma,u\#A)+(\Phi\partial+\partial\Phi)(\gamma,u)$.
    \end{center}
    That is because $(a\cup)^{n+1}$ is homotopic to $(\cup^Qa)^{n+1}$ by some $\Phi:CF^*\rightarrow CF^*[-1]$. If the right hand side is non-zero, there must exist some $u_1,u_2,\dots,u_{n+1} $ such that $(\gamma',u\#u_1\#u_2\#\cdots\#u_{n+1})$ is a term in the left hand side of the above equality. We may assume $u_i$ start with $\beta_{i-1}$ and end with $\beta_i$. $\beta_0 = \gamma,\beta_{n+1} = \gamma' $. Then we claim there exist a $\gamma$ such that $\beta_i$ lie in different groups for $i = 1,2,\cdots,n+1$, so there must be at least $n+1$ groups. Moreover, $\psi_H$ has at least $n+1$ 1-periodic orbits. To prove this, we need a lemma 
    \begin{lemma}\label{lemma}

H is a Hamiltonian, $[A]\in H_2(M)$ expressed by a 2-cycle A in M which do not intersect with any 1-periodic orbits of $\psi_H$. Then there exist a positive number $\delta$ and a small neighborhood U of H in $C^{\infty}$ sense. Such that for any $H'\in U$, $u:\mathbb{R}_\tau\times S^1_t\rightarrow M,\ \bar\partial_{H'}u = 0$ with finite energy, $u(0,0)\in A$. We have $\int_{\mathbb{R}\times S^1}||\partial_\tau u||^2>\delta$.

    \end{lemma}

    This is easy to imagine that if u has finite energy, it must start and end with some 1-periodic orbits. Then if it hits A at point $u(0,0)$, it will have a lower bounded length and thus a lower bounded energy. We will prove it at last.

    We now assume for any $(\gamma,u)$ and non-trivial $(\gamma',u\#u_1\#u_2\#\cdots\#u_{n+1})$, we must have $\beta_i$ and $\beta_j$ lie in the same group for some $i \neq j$. Otherwise we already get $n+1$ periodic loops. Then we choose N as a positive integer. N larger than the number of 1-periodic orbits of $\psi_{\tilde{H}}$. Because $((a\cup)^{n+1})^N$ is homotopic to $(e^A)^N$ which is non-trivial. There must be some $(\gamma,u)\in CF^*(M)$ s.t. $((a\cup)^{n+1})^N(\gamma,u)\neq 0$. Let $(\gamma',u\#u_{1,1}\#\cdots\#u_{1,n+1}\#\cdots\#u_{N,n+1})$ be a nontrivial term in it and $u_{j,n+1}$ end with $l_j$ $j = 1,2,\dots,N$. Because N is large enough, so we must have $l_i = l_j $ for some $i\neq j$.

    By our assumption, for any $(l_k,u_k)$ and $(l_{k+1},u_{k+1})$, $u_k = u\#u_{1,1}\#\cdots\#u_{k,n+1}$ and $u_{i,j}$ end with $l_{i,j}$, we must have $l_{k+1,i}$ and $l_{k+1,j}$ lie in the same group for some $i\neq j$. Note that $\int_{\mathbb{R}\times S^1}u^*\omega = \int_{x^-}H_t-\int_{x^+}H_t + \int_{\mathbb{R}\times S^1}||\partial_\tau u||^2$ where $x^-= \lim\limits_{\tau\rightarrow-\infty}u(\tau,\cdot)$ $x^+= \lim\limits_{\tau\rightarrow+\infty}u(\tau,\cdot)$, so 
    \begin{center}
            $\omega(u_{k+1,i+1}\#\cdots\#u_{k+1,j}) = \int_{l_{k+1,i}}H_t-\int_{l_{k+1,j}}H_t + \sum\limits_{l = i+1}^{j}\int_{\mathbb{R}\times S^1}||\partial_\tau u_{k+1,l}||^2$.
    \end{center}

\begin{figure}[t]
\begin{center}
\includegraphics[width=0.6\textwidth]{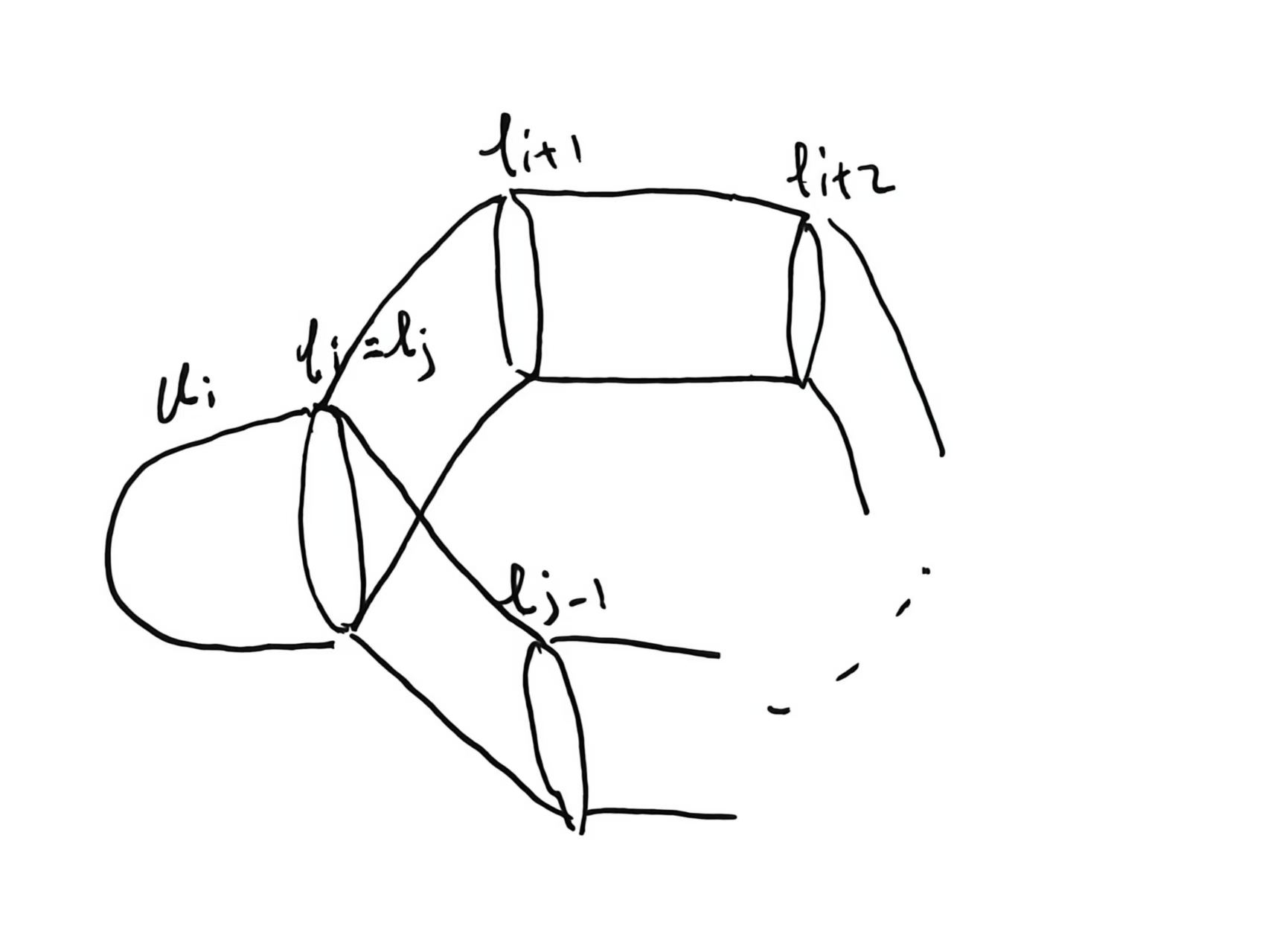}
\caption{The nontrivial trajectory.}
\label{fig:loop}
\end{center}
\end{figure}

\begin{figure}[t]
\begin{center}
\includegraphics[width=0.6\textwidth]{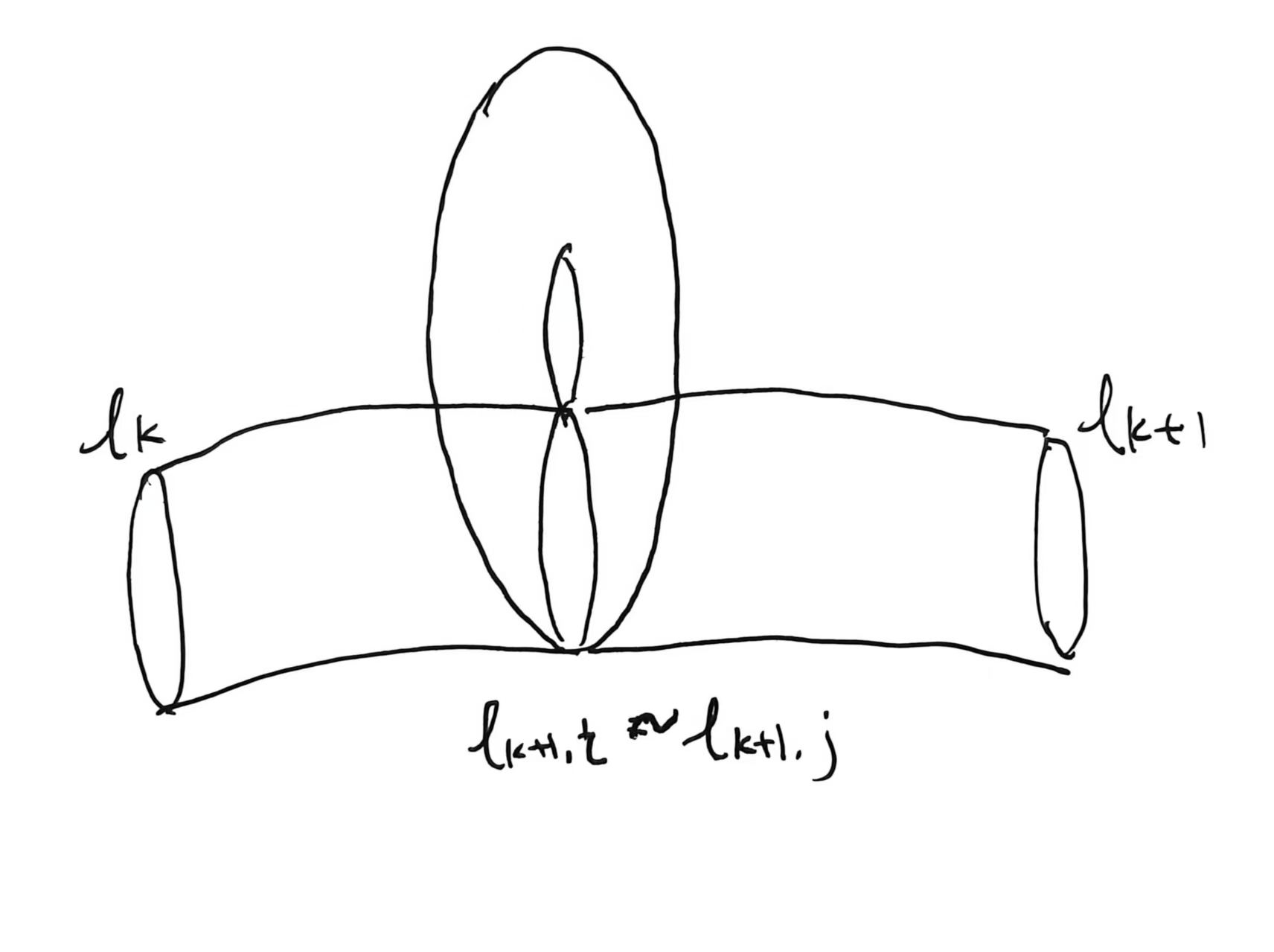}
\caption{The torus appears between $l_i$ and $l_{i+1}.$}
\label{fig:torus}
\end{center}
\end{figure}
    
    Because $l_{k+1,i}$ and $l_{k+1,j}$ lie in the same group, $|\int_{l_{k+1,i}}H_t-\int_{l_{k+1,j}}H_t|$ should be smaller than $\delta/3$ with a suitable $\tilde{H}$. And we can use a cylinder $v:\mathbb{R}\times S^1\rightarrow M$ with symplectic energy less than $\delta/3$ to join $l_{k+1,i}$ and $l_{k+1,j}$. Use lemma above we can get 
    \begin{center}
        $\omega(v\#u_{k+1,i+1}\#\cdots\#u_{k+1,j})\geq = \sum\limits_{l = i+1}^{j}\int_{\mathbb{R}\times S^1}||\partial_\tau u_{k+1,l}||^2 - |\int_{l_{k+1,i}}H_t-\int_{l_{k+1,j}}H_t| - |\omega(v)|$\\
        $\geq \delta - \delta/3-\delta/3 > 0$.
    \end{center}
    And in $\mathbb{C}P^n$ case, this means $\omega(v\#u_{k+1,i+1}\#\cdots\#u_{k+1,j})\geq p$   where $\omega(H_2(\mathbb{C}P^n)) = p\cdot \mathbb{Z}$. Now we find that \\
    \begin{align}
        \sum\limits_{l = 1}^{n+1}\int_{\mathbb{R}\times S^1}||\partial_\tau u_{k+1,l}||^2\geq & \sum\limits_{l = i+1}^{j}\int_{\mathbb{R}\times S^1}||\partial_\tau u_{k+1,l}||^2 + \delta \nonumber\\
        \geq & \delta + \omega(u_{k+1,i+1}\#\cdots\#u_{k+1,j}) - |\int_{l_{k+1,i}}H_t-\int_{l_{k+1,j}}H_t| \nonumber \\
        = & \delta +\omega(v\#u_{k+1,i+1}\#\cdots\#u_{k+1,j})-\omega(v) \nonumber\\
        &-|\int_{l_{k+1,i}}H_t-\int_{l_{k+1,j}}H_t| \nonumber \\
        \geq & \delta + p - \delta/3 -\delta/3>p
    \end{align}  
    Come back to the case $l_i = l_j$, each $\alpha\cup$ is a deg 2 operator, so $\mu(l_j,u_j) = \mu(l_i,u_1) + 2(n+1)(j-i)$. On the other hand, $u_j = u_i\#u_{i+1,1}\#\cdots\#u_{j,n+1}$, so $\mu(l_j,u_j) = \mu(l_i,u_1) + 2c_1(u_{i+1,1}\#\cdots\#u_{j,n+1})$. That means $c_1(u_{i+1,1}\#\cdots\#u_{j,n+1}) = (n+1)(j-1)$. Because $\mathbb{C}P^n$ is monotone, we know $\omega(u_{i+1,1}\#\cdots\#u_{j,n+1}) = p(j-i)$. However,  $\omega(u_{i+1,1}\#\cdots\#u_{j,n+1}) = \sum\limits_{k = i}^{j-1}\sum\limits_{l = 1}^{n+1}\int_{\mathbb{R}\times S^1}||\partial_\tau u_{k+1,l}||^2 > p(j-i)$ by (1), so we get the contradiction and end the proof.
\end{proof}

We should notice that in Schwarz's work, he used points, which is  in general homology level, to represent elements in the base chain complex and used $\bar{\partial}_H$-holomorphic trajectories to describe the quantum cup product structure. And our method here use $\bar{\partial}_H$-holomorphic trajectories to compute the base chain complex and use classes in general homology to describe the quantum structure. The key idea of these two methods is the same. That is, to estimate the number of different periodic loops in this nontrivial trajectory. 

Now, to finish our proof, we should prove  \ref{lemma}, and we will prove a claim first

\begin{claim}
    For $u:\mathbb{R}_s\times S^1_t\rightarrow M$ with $\bar{\partial}u = 0$. Let $\omega := ||\partial_\tau u||^2$, then we have $\Delta \omega\geq -a\omega^2-b$ for some constant a, b depend on M, J, H. Especially, they depend on H continuously.
\end{claim}
    
\begin{proof}
    Let $\xi = \partial_su,\ \eta = \partial_t u$, then we have relation $\xi+J\eta+\nabla H = 0$. $\nabla$ is Levi-Civita connection correspond to the metric induced by complex structure and symplectic form. $\omega = \xi^2$ and 

    \begin{align*}
        \frac{1}{2}\Delta\omega = & ||\nabla_s\xi||^2 + ||\nabla_t\xi||^2 + 2<\xi,\nabla_s\nabla_s\xi+\nabla_t\nabla_t\xi>
    \end{align*}
    \begin{align*}
        \Delta\xi = & \nabla_s\nabla_s\xi+\nabla_t\nabla_t\xi\\
         = &\nabla_s(\nabla_s\xi+\nabla_t\eta)+\nabla_t\nabla_s\eta-\nabla_s\nabla_t\eta\\
          = & \nabla_s(\nabla_s\xi+\nabla_t\eta) + R(\eta,\xi)\eta
    \end{align*}
    $|R(\eta,\xi)\eta|\leq a\|du\|^3$ for some $a$ depend on $M, J$. 
    \begin{align*}
        \nabla_s(\nabla_s\xi+\nabla_t\eta) = & \nabla_s(\nabla_s(-J\eta-\nabla H)+\nabla_t (J\xi+J\nabla H))\\
         = &\nabla_s(\nabla_t(J)\xi-\nabla_s(J)\eta+\nabla_t(J\nabla H)-\nabla_s(\nabla H))\\
          = & \nabla_s(\nabla_t(J)) + \nabla_t(J)\nabla_s\xi - \nabla_s(\nabla_s(J)) \\
          &-\nabla_s(J)\nabla_s\eta+\nabla_t(J\nabla H)-\nabla_s(\nabla H)
    \end{align*}
    $\nabla_s(\nabla_t(J))$ and $\nabla_s(\nabla_s(J))$ can be controlled by $c(\xi^2+\eta^2+|\nabla_s\xi|+|\nabla_t\xi|)$. $\nabla_t(J)\nabla_s\xi$ and $\nabla_s(J)\nabla_s\eta$ can be controlled by $c(|\xi|+|\eta|)(|\nabla_s\xi|+|\nabla_t\xi|)$. The last two terms can be controlled by $c(|\xi|+|\eta|)$. c is a constant depends on M, J and H. Moreover, it depends on H continuously. So
    \begin{align*}
       | <\xi,\nabla_s\nabla_s\xi+\nabla_t\nabla_t\xi>|\leq & c|\xi|\cdot(|du|^3+(\xi^2+\eta^2+|\nabla_s\xi|+|\nabla_t\xi|)\\
       +&(|\xi|+|\eta|)(|\nabla_s\xi|+|\nabla_t\xi|)+(|\xi|+|\eta|))\\
       \leq & c'|du|^4+\epsilon(|\nabla_s\xi|^2+|\nabla_t\xi|^2)+d
    \end{align*}
    $\epsilon$ is a small number and $c'$, $d$ are constants depend on M, J and H. Of course, they depend on H continuously. So we get 
    \begin{center}
        $\frac{1}{2}\Delta\omega\geq -a\|du\|^4-b.$
    \end{center}
 Using the relation $\xi+J\eta+\nabla H = 0$, we know $\|du\|^2 = \|\partial_s u\|^2+\|\partial_su+\nabla H\|^2\leq 3\|\partial_su\|^2+2\|\nabla H\|^2$. So we can get the result we need.
\end{proof}

Now let $u:\mathbb{R}\times S^1\rightarrow M$ with $u(0,0)\in A$, $\omega = \|\partial_\tau u\|^2$. Let $f(\rho) := (1-\rho)^2\sup_{B_\rho}\omega$ for $0\leq \rho \leq 1$, $f(1) = 0$ and $f\geq 0$. So there exist $\rho^*\in [0,1)$ such that $f(\rho^*)$ is the maximum value of f in [0,1]. Then we denote $\sup_{B_{\rho^*}}(\omega) = \omega(z^*)$ by c for some $z^*\in B_1 \subset \mathbb{R}\times S^1$.

Let $\epsilon:=\frac{1-\rho^*}{2}$, we have $\sup_{B_\epsilon(z^*)}\omega\leq\sup_{B_{\rho^*+\epsilon}}\omega = \frac{f(\rho^*+\epsilon)}{(1-\rho^*-\epsilon)^2} = \frac{4f(\rho^*+\epsilon)}{(1-\rho^*)^2}\leq \frac{4f(\rho^*)}{(1-\rho^*)^2} = 4c$ and thus $\Delta\omega\geq -a\omega^2-b\geq-16ac^2-b$ in $B_\epsilon(z^*)$.

Note that for any smooth $\omega:B_r\rightarrow\mathbb{R}$, $\Delta\omega\geq-m$, m is a constant, then $\omega(0)\leq \frac{mr^2}{8}+\frac{1}{\pi r^2}\int_{B_r}\omega$. This is mean value inequality of $\nu = \omega+\frac{m}{4}(x^2+y^2)$. In above case, it is $c = \omega(z^*)\leq \frac{16ac^2+b}{8}\rho^2+\frac{1}{\pi\rho^2}\int_{B_\rho(z^*)}\omega$ for $0<\rho\leq \epsilon$. Then we assume that $\int_{B_1}\omega<\frac{\pi}{16a}$

1) If $4ac\epsilon^2\geq1$, we take $\rho = \sqrt{\frac{1}{4ac}}  \leq \epsilon$. Then $c = \omega(z^*)\leq \frac{c}{2}+\frac{b}{32ac}+\frac{4ac}{\pi}\int_{B_1}\omega$ and use $\int_{B_1}\omega<\frac{\pi}{16a}$, we get $c\leq \frac{c}{2}+\frac{b}{32ac}+\frac{c}{4}$. Thus $c\leq\sqrt\frac{b}{8a}$ and $\omega(0) = f(0) \leq f(\rho^*) = (1-\rho^*)^2c\leq c \leq \sqrt{\frac{b}{8a}}$.

2) If $4ac\epsilon^2<1$, $\omega(0) = f(0)\leq f(\rho^*) = (1-\rho^*)^2c = 4c\epsilon^2<\frac{1}{a}$.

So we get if $\int_{B_1}\omega<\frac{\pi}{16a}$, then $\omega(0)<\min\{\frac{1}{a},\sqrt{\frac{b}{8a}}\}$. Denote this bound by $C(a,b)$. Recall that we are dealing with $u:\mathbb{R}\times S^1\rightarrow M$ and $\omega = \|\partial_\tau u\|^2$, then if $\int_{\mathbb{R}\times S^1}\omega<\frac{\pi}{16a}$, we will get $|\partial_\tau u|<C(a,b)$ because for any point in $\mathbb{R}\times S^1$ we can choose a unit disc in $\mathbb{R}\times S^1$ centered at this point and restrict the problem on this disc. Moreover, $|\partial_t u| = |\partial_\tau u +\nabla H|\leq C'(a,b) = C(a,b)+\max|\nabla H|$, so $|du|\leq 2C'(a,b)$

Because $u(0,0)\in A$, $u(B_{\frac{\epsilon}{4C'}})\subset B_{\frac{\epsilon}{2}}(A)$. Here $B_{\frac{\epsilon}{2}}(A):=\{x\in M|d(x,A)\leq \frac{\epsilon}{2}\}$ and $\epsilon$ is a positive number such that $B_\epsilon(A)$ do not intersect with 1-periodic orbits of $\psi_H$. Moreover, we can get $u([-\frac{\epsilon}{4C'},\frac{\epsilon}{4C'}]\times{0})\subset B_{\frac{\epsilon}{2}}(A)$ and we will finish the proof by the following lemma

\begin{lemma}
    There exist $\delta>0$ depends on M, J, A and H such that for any $\gamma:S^1\rightarrow M$, $\gamma(0)\in B_{\frac{\epsilon}{2}}(A)$, we have $\int_{S^1}\|J\gamma'+\nabla H\|^2>\delta$. Of course, $\delta$ depends on H locally continuously. Here, the locally means it is continuous in a small neighborhood of H.
\end{lemma}

Using this lemma, if $\int_{B_!}\omega<\frac{\pi}{16a}$, $\int_{\mathbb{R}\times S^1}|\partial_\tau u|^2>\int_{[-\frac{\epsilon}{4C'},\frac{\epsilon}{4C'}]\times S^1}|\partial_\tau u|^2 = \int_{[-\frac{\epsilon}{4C'},\frac{\epsilon}{4C'}]\times S^1}|J\partial_t u +\nabla H|^2>\frac{\epsilon}{2C'}\cdot \delta$, so we finished the proof of \ref{lemma}.

\begin{proof}
    There exists $\delta_1$ such that any $\gamma:S^1\rightarrow M$, $\gamma(0)\in B_{\frac{\epsilon}{2}}(A)$, it should satisfy $d(\gamma(0),\psi_H^1(\gamma(0)))\geq \delta_1$ and  $d(\gamma(0),\psi_H^{-1}(\gamma(0)))\geq \delta_1$.
    
    Take $N$ a large integer with $N\geq 2\|\nabla(J\nabla H)\|_{L^\infty}$. Define a path $\beta(t):= \psi_H^{-t}(\gamma(t))$ between $\gamma(0) $ and $ \psi_H^{-1}(\gamma(1))$. Choose some points $\gamma_1,\gamma_2,\dots ,\gamma_{N-1}$ on $\beta$ such that $d(\gamma_i,\gamma_{i+1})\geq \frac{1}{N}d(\gamma_0,\gamma_N)$, where $\gamma_0 = \gamma(0)$ and $\gamma_N = \psi_H^{-1}(\gamma(1))$. Now consider a function $f:M\times M\rightarrow \mathbb{R}$, $f(x,y) = \frac{\inf\limits_{0\leq t \leq 1}d(\psi_H^t(x),\psi_H^t(y))}{d(x,y)}$, $f$ is positive and continuous on $M\times M$ except the diagonal $\Delta$. So f has a minimal value $\delta_2$ on $M\times M - B_{\epsilon'}(\Delta)$, here $\epsilon'$ is a positive number smaller than $\frac{\delta_1}{N}$, we can choose it to be $\frac{\delta_1}{2N}$ for simplicity. So $\delta_2$ is depend on $\delta_1$ and $N$ continuously, thus it depends on H locally continuously.

\begin{figure}[t]
\begin{center}
\includegraphics[width=0.8\textwidth]{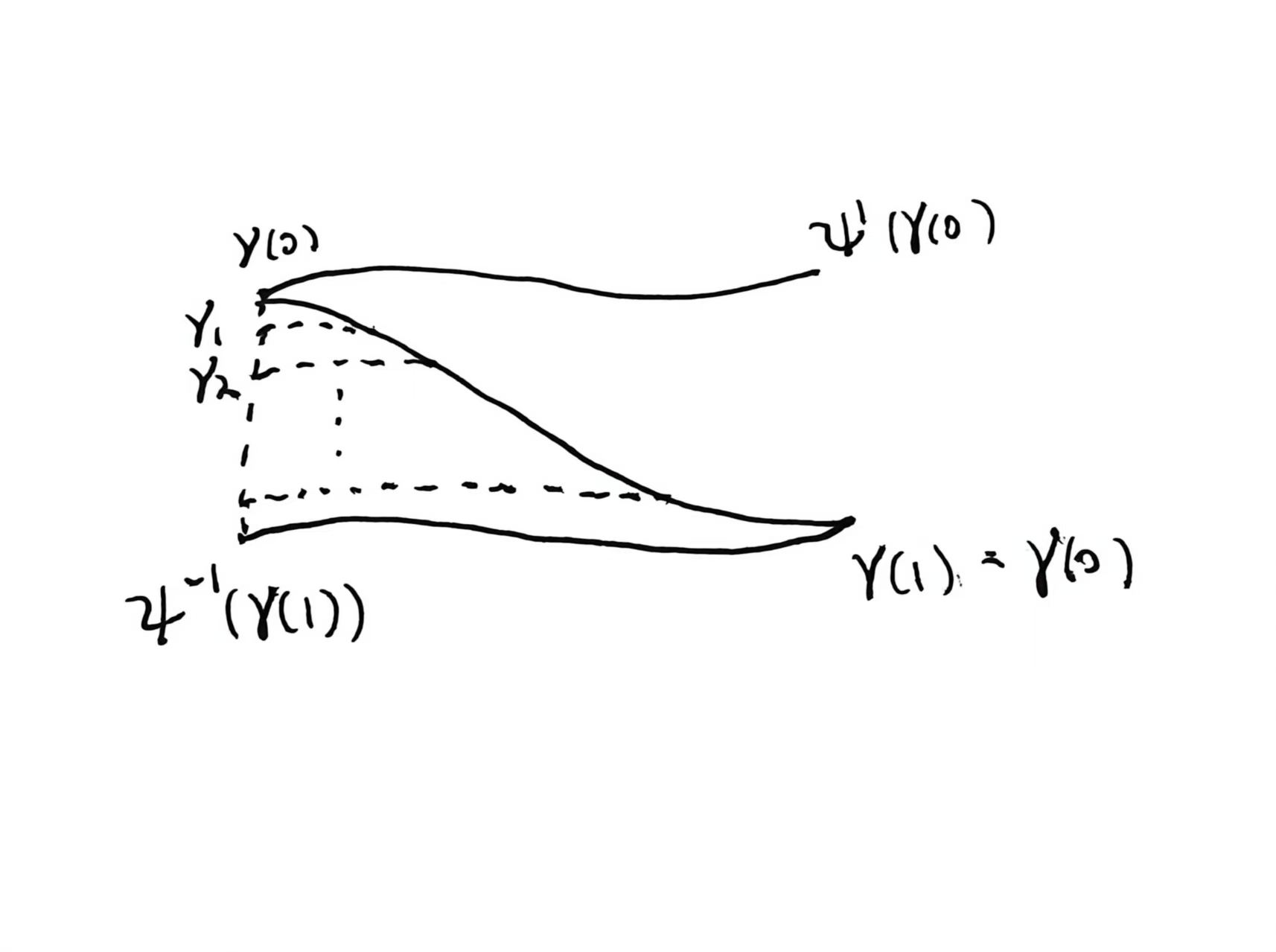}
\caption{Path from $\gamma(0)$ to $\gamma(1)$.}
\label{fig:mainlemma}
\end{center}
\end{figure}

    Let $t_i = \beta^{-1}(\gamma_i)$, then there must exist some i such that $t_{i+1}-t_1\geq \frac{1}{N}$. Use the definition we can get 
    \begin{align*}
        d:=&d(\gamma(t_{i+1}),\psi_H^{t_{i+1}-t_i}\gamma(t_i)) \\
          =&d(\psi_H^{t_{i+1}}\gamma_{i+1},\psi_H^{t_{i+1}}\gamma_i)\\
          \geq & \delta_2\cdot d(\gamma_{i+1},\gamma_i)\\
          \geq & \delta_2\cdot \frac{1}{N}\cdot d(\gamma_0,\gamma_N)\\
          \geq & \delta_2\cdot \frac{1}{N}\cdot\delta_1.
    \end{align*}

    We take derivative of $d(\gamma(t_i+t),\psi_H^t(\gamma(t_i)))$
    \begin{align*}
        \frac{d}{dt}&(d(\gamma(t_i+t),\psi_H^t(\gamma(t_i)))) = \dot{\gamma}(t_i+t)\cdot X_t+\dot{\psi_H^t}(\gamma(t_i))\cdot Y_t\\
         = & (\dot{\gamma}(t_i+t)-J\nabla H(\gamma(t_i+t)))\cdot X_t + J\nabla H(\gamma(t_i+t))\cdot X_t +J\nabla H(\psi_H^t(\gamma(t_i)))\cdot Y_t
    \end{align*}
    Let $s$ be the minimal geodesic between $\gamma(t_i+t)$ and $\psi_H^t(\gamma(t_i))$, $s(0) = \gamma(t_i+t)$ and $s(l) = \psi_H^t(\gamma(t_i))$. Then $X_t$ is $-\dot{s}(0)$ and $Y_t$ is $\dot{s}(l)$. Namely, $X_t = \nabla d(\cdot,\psi_H^t(\gamma(t_i)))(\gamma(t_i+t))$, $Y_t = \nabla d(\gamma(t_i+t),\cdot)(\psi_H^t(\gamma(t_i)))$. 

    $|J\nabla H(\gamma(t_i+t))\cdot X_t +J\nabla H(\psi_H^t(\gamma(t_i)))\cdot Y_t| = |\int_s \frac{d}{dt}(J\nabla H\cdot \dot{s}(t))|\leq \|\nabla(J\nabla H)\|\cdot length(s)\leq \frac{N}{2}\cdot length(s)$. We now take $\tilde{t_i} = \inf\limits_{t_i\leq t}\{t|d(\gamma(t),\psi_H^{t-t_i}(\gamma(t_i)))\geq d\}$ and $|J\nabla H(\gamma(t_i+t))\cdot X_t +J\nabla H(\psi_H^t(\gamma(t_i)))\cdot Y_t|\leq \frac{N}{2}\cdot length(s)\leq \frac{N}{2}\cdot d$ for $t\in [0,\tilde{t_i}-t_i]$. So
    \begin{align*}
        d = & d(\gamma(\tilde{t_i}),\psi_H^{\tilde{t_i}-t_i}(\gamma(t_i)))\\
        \leq & \int_{t_i}^{\tilde{t_i}}|\frac{\partial}{\partial t}d(\gamma(t),\psi_H^{t-t_i}(\gamma(t_i)))|dt\\
        \leq & \int_{t_i}^{\tilde{t_i}}(|\dot{\gamma}(t)-J\nabla H|+d\cdot \frac{N}{2})dt\\
        \leq & \int_{t_i}^{\tilde{t_i}}(|\dot{\gamma}(t)-J\nabla H|)dt+d\cdot\frac{N}{2}\cdot\frac{1}{N}.
    \end{align*}
    So we get $\frac{d}{2}\leq \int_{t_i}^{\tilde{t_i}}(|\dot{\gamma}(t)-J\nabla H|)dt$, thus $\int_{0}^{1}(|\dot{\gamma}(t)-J\nabla H|^2)dt\geq\int_{0}^{1}(|\dot{\gamma}(t)-J\nabla H|)dt\geq \frac{d}{2}\geq \frac{\delta_1\delta_2}{2N}$.
\end{proof}

\subsection{Application on Monotone 1-point Blow up of $\mathbb{C}P^n$}

\begin{proposition}
    This estimate also hold for monotone 1-point Blow up of $\mathbb{C}P^n$
\end{proposition}
\begin{proof}
    To extend this result on monotone 1-point Blow up of $\mathbb{C}P^n$, we first need to clear the quantum cohomology structure of it. Denote this manifold by M, then $H^*(M) = \mathbb{Z}[a,b]/(a^{n+1} = b^{n+1} = 0,a^n = b^n)$ and $\deg(a) = \deg(b) = 2$. The Novikov ring $\Lambda$ is cmposed of 
    \begin{center}
        $\{\sum a_ie^{k_iA} |\ k_i\rightarrow\infty\ \text{as}\ i\rightarrow\infty \}$.
    \end{center}
    Here, $A\in H_2(M)$ is a generator with the minimal positive symplectic energy on $H_2(M)$. And it's sufficient to prove $(\cup^Q)^{n+1}a = e^A$ in $HQ^*(M)$ so we can repeat the process in case of $\mathbb{C}P^n$.
    
    For $a^m\cup^Q a = a^m*_0a+\sum\limits_{i\neq0}a_{i}e^{iA}$, if $0\leq m<n$, it will only have $a^m*_0a$ part because $\deg(a_{i}e^{iA}) = 2ic_1(A) + \deg(a_i)= 2i(n+1)+\deg(a_i)$. It must greater or equal to $2(n+1)$ because the moduli space $\mathcal{M}(iA)$ should be nonzero. But $\deg(a^m\cup^Q a) = 2(m+1)<2(n+1)$. Namely, it's just the usual cup product and $a^m\cup^Q a = 0$. 
    
    If $m = n$,  $a^n\cup^Qa = a^{n+1}+a^n*_Aa = a^n*_Aae^A$ by the similar argument of degree issue. By definition, $a^n*_Aa = GW_{0,3}^{M,A}(a^n,a,dvol)$. There is exactly one holomorphic sphere intersect with $PD[a^n] = pt$, $PD[a]$, $PD[dvol] = pt$ at 1, 0, $\infty$. Namely, $GW_{0,3}^{M,A}(a^n,a,dvol) = 1$, thus $a^n\cup^Qa = e^A$.
\end{proof}

The monotonicity is important because it is used in energy estimation. for $\mathbb{C}P^n$ blow up at several points, if it has a monotone symplectic form $\omega$, the estimate above still hold because the relation $(\cup^Q)^{n+1}a = e^A$ still holds in this case.

In fact, this method will work for all manifolds whose quantum cohomology has an element $a\in H^2(M)$, and some power of $a$ in $HQ^*(M)$ has a factor $e^A$, $A\in H_2(M)$ has the minimal symplectic energy.

As an example, for $\prod\mathbb{CP}^{i_k}$. If the symplectic energy of each factor are the same, we can get $\max\{i_k+1\}$ fixed points. This is coinside with the result of Givental, and is stronger than Floer said in his original paper.

\newpage
\bibliographystyle{plain}
\bibliography{references.bib}
\end{document}